\documentclass[10pt,oneside,a4paper]{article} 	

\usepackage[fleqn]{amsmath} 									
\usepackage{amsthm,amsfonts,latexsym,amssymb,amscd} 	
\usepackage[mathscr]{eucal}   									
\usepackage[all]{xy}														
\usepackage[draft=false]{hyperref} 

\newcommand{\M}{\mathcal{M}}

\newcommand{\Gg}{\mathfrak{G}}

\newcommand{\CC}{{\mathbb{C}}}

\newcommand{\NN}{{\mathbb{N}}}

\newcommand{\As}{{\mathscr{A}}}\newcommand{\Bs}{{\mathscr{B}}}\newcommand{\Cs}{{\mathscr{C}}}

\newcommand{\Is}{{\mathscr{I}}}\newcommand{\Ks}{{\mathscr{K}}}
\newcommand{\Ms}{{\mathscr{M}}}
\newcommand{\Rs}{{\mathscr{R}}}

\DeclareFontFamily{U}{rsfs}{\skewchar\font127 }
\DeclareFontShape{U}{rsfs}{m}{n}{%
   <5> <6> rsfs5
   <7> rsfs7
   <8> <9> <10> <10.95> <12> <14.4> <17.28> <20.74> <24.88> rsfs10
}{}
\DeclareSymbolFont{rsfs}{U}{rsfs}{m}{n}
\DeclareSymbolFontAlphabet{\scr}{rsfs}

\newcommand{\Cf}{\scr{C}}\newcommand{\Df}{\scr{D}}

\DeclareMathOperator{\spa}{span}

\DeclareMathOperator{\Ob}{Ob}

\DeclareMathOperator{\Hom}{Hom}

\DeclareMathOperator{\Sp}{Sp}

\DeclareMathOperator{\End}{End}

\DeclareMathOperator{\Pic}{Pic}

\renewcommand{\emph}{\textbf} 										
\newcommand{\cj}[1]{\overline{#1}}								
\newcommand{\ip}[2]{\langle #1\mid #2\rangle}			
\renewcommand{\iff}{\Leftrightarrow}								
\newcommand{\imp}{\Rightarrow}										
\newcommand{\cs}{C*}

\newcommand{\hlink}[2]{\href{#1}{\texttt{#2}}} 

\newtheorem{theorem}{Theorem}[section]			

\newtheorem{lemma}[theorem]{Lemma}
\newtheorem{proposition}[theorem]{Proposition}

\newtheorem{definition}[theorem]{Definition}
\newtheorem{remark}[theorem]{Remark}

\numberwithin{equation}{section}  	
\title{\textbf{A Spectral Theorem for Imprimitivity C*-bimodules}}

\author{\normalsize Paolo Bertozzini \footnote{Partially supported by the Thai Research Fund: grant n.~RSA4780022.} @, 
Roberto Conti $^*\ddag$, Wicharn Lewkeeratiyutkul $^*\S$
\\
\normalsize @ e-mail: \texttt{paolo.th@gmail.com}
\\ 
\normalsize $\ddag$ \textit{Mathematics, School of Mathematical and Physical Sciences,} 
\\ 
\normalsize \textit{University of Newcastle, Callaghan, NSW 2308, Australia}
\\
\normalsize  e-mail: \texttt{Roberto.Conti@newcastle.edu.au} 
\\
\normalsize $\S$ \textit{Department of Mathematics, Faculty of Science,}
\\
\normalsize \textit{Chulalongkorn University, Bangkok 10330, Thailand}
\\ 
\normalsize  e-mail: \texttt{Wicharn.L@chula.ac.th}
}

\date{\normalsize{08 December 2008}}

\begin{document}

\maketitle

\begin{abstract} \noindent 
After recalling in detail some basic definitions on Hilbert C*-bimodules, Morita equivalence and imprimitivity, we discuss a spectral reconstruction theorem for imprimitivity Hilbert C*-bimodules over commutative unital C*-algebras and consider some of its applications in the theory of commutative full C*-categories.   

\medskip

\noindent 
MSC-2000: 
					46L08,		
					46L87,		
					16D90,		
	  				46M20. 		

\medskip

\\
Non-commutative Geometry. 
\end{abstract}

\section{Introduction} 
A.~Connes' non-commutative geometry~\cite{C} is the most powerful incarnation of R.~Descartes' idea of trading ``geometrical spaces'' with commutative ``algebras of coordinates'' and it is based on the existence of suitable dualities between categories constructed from commutative algebras and categories of their ``spectra''. 
The most celebrated example is I.~Gel'fand-M.~Na\u\i mark theorem (see e.g.~\cite[Theorem~II.2.2.4]{B}) asserting that, via Gel'fand transform, a unital commutative C*-algebra $\As$ is isomorphic to the algebra of continuous complex-valued functions on a compact Hausdorff topological space, namely the spectrum of $\As$. 
In this way a commutative unital \hbox{C*-algebra} can be reconstructed (up to isomorphism) from its spectrum. 

The equally famous Serre-Swan theorem (see e.g.~\cite[Theorem~6.18]{Kar}) permits the reconstruction, up to isomorphism, of a finite projective module over a commutative unital \hbox{C*-algebra} from a spectrum that turns out to be a finite-rank complex vector bundle over the Gel'fand spectrum of the C*-algebra. 
When we restrict to the case of Hilbert C*-modules over commutative unital C*-algebras, Serre-Swan theorem admits a more powerful 
formulation, Takahashi theorem~\cite{Ta1,Ta2,W}, with spectra given by Hilbert bundles over compact Hausdorff spaces. 

The purpose of this paper is to start the development of a spectral reconstruction theorem for suitable bimodules over commutative unital C*-algebras, i.e.~a ``bivariant version'' of Takahashi and Serre-Swan results, considering for now the case of imprimitivity Hilbert C*-bimodules. 

In order to make the result almost completely self-contained, we precede the discussion of our spectral theorem with a detailed treatment of basic facts on imprimitivity C*-bimodules and Morita equivalence including an explicit construction of a natural isomorphism between a pair of C*-algebras associated to a given imprimitivity Hilbert C*-bimodule over them. 

Our main result is that the spectrum of an imprimitivity Hilbert C*-bimodule over two commutative unital C*-algebras is described by a Hermitian line bundle over a compact Hausdorff space that is the graph of a canonical homeomorphism between the Gel'fand spectra of the two unital C*-algebras i.e.~every imprimitivity Hilbert C*-bimodule is isomorphic to a suitably twisted bimodule of sections of this ``spectral'' Hermitian line bundle. 

We will also collect together some facts about imprimitivity C*-bimodules in the setting of C*-categories that provide a useful background for our study of a categorical Gel'fand duality~\cite{BCL2} and that cannot be easily found in the literature. 

\medskip

The content of the paper is as follows. 

In section~\ref{sec: modules}, for the benefit of the readers, we recall the basic definitions and properties of Hilbert C*-modules. 
In subsection~\ref{sec: imp-bim-cat} we explore some specific properties of imprimitivity bimodules arising from C*-categories that will be crucial in the study of the categorification of Gel'fand duality that will be undergone in~\cite{BCL2}. 
Section~\ref{sec: bimodule} contains the proof of the spectral reconstruction theorem for imprimitivity Hilbert C*-bimodules as well as some relevant bibliographical references to other available spectral results for C*-modules. 

The complete construction of a bivariant duality, between categories of ``bivariant Hermitian (line) bundles'' and categories of (imprimitivity) Hilbert C*-bimodules over commutative unital C*-algebras, will not be completed here (in particular there is no discussion of the appropriate classes of morphisms and no construction of the section/spectrum functors supporting such a duality), but it is our intention to return later to this topic. 

Part of the results presented here have been announced in our survey paper~\cite{BCL} and have been presented in several seminars in Thailand, Australia, Italy, UK since May 2006. 

\section{Preliminaries on Hilbert C*-Modules}
\label{sec: modules}

For convenience of the reader and in order to establish notation and terminology, we provide here some background material on the theory of Hilbert \cs-modules. General references are the books by N.~Wegge-Olsen~\cite{WO}, C.~Lance~\cite{La} and B.~Blackadar~\cite[Section~II.7]{B}. 

\smallskip

In the following, $\As,\Bs, \dots$ denote unital \cs-algebras and $\As_+:=\{a^*a\in \As \ | \ a\in \As\}$ is the positive part of the \cs-algebra $\As$.  
\begin{definition}
a \emph{right pre-Hilbert \cs-module} $M_\Bs$ over a unital \cs-algebra $\Bs$ is a unital right module over the unital ring $\Bs$ that is equipped with a $\Bs$-valued inner product $(x,y)\mapsto \ip{x}{y}_\Bs$ such that:
\begin{gather*}
\ip{z}{x+y}_\Bs=\ip{z}{x}_\Bs + \ip{z}{y}_\Bs \quad \forall x,y,z\in M, \\
\ip{z}{x\cdot b}_\Bs=\ip{z}{x}_\Bs b \quad \forall x,y \in M, \ \quad \forall b\in \Bs, \\ 
\ip{y}{x}_\Bs=\ip{x}{y}_\Bs^* \quad \forall x,y\in M, \\
\ip{x}{x}_\Bs\in \Bs_+ \quad \forall x\in M, \\
\ip{x}{x}_\Bs=0_\Bs \imp x=0_M. 
\end{gather*}

Analogously, a \emph{left pre-Hilbert \cs-module} ${}_\As M$ over a unital \cs-algebra $\As$ is a unital left module $M$ over the unital ring $\As$, 
that is equipped with an $\As$-valued inner product $M\times M\to \As$ denoted by $(x,y)\mapsto {}_\As\ip{x}{y}$. Here the $\As$-linearity in on the first variable.
\end{definition}

\begin{remark} 
A right (respectively left) pre-Hilbert \cs-module $M_\Bs$ over the \cs-algebra $\Bs$ is naturally equipped with a norm (for a proof see for 
example~\cite[Lemma~2.14 and Corollary~2.15]{FGV}): 
\begin{equation*}
\|x\|_M:=\sqrt{\|\ip{x}{x}_\Bs\|_\Bs}, \quad \forall x \in M.  
\end{equation*}
\end{remark}

\begin{definition}
A right (resp.~left) \emph{Hilbert \cs-module} is a right (resp.~left) pre-Hilbert \cs-module over a \cs-algebra $\Bs$ that is a Banach space 
with respect to the previous norm $\|\cdot\|_M$ (resp.~${}_M\|\cdot\|$). 
\end{definition}

\begin{definition}
A right Hilbert \cs-module $M_\Bs$ is said to be \emph{full} if 
\begin{equation*}
\ip{M_\Bs}{M_\Bs}_\Bs:=\cj{\spa\{\ip{x}{y}_\Bs\ | \ x,y\in M_\Bs\}}= \Bs, 
\end{equation*}
where the closure is in the norm topology of the \cs-algebra $\Bs$.
A similar definition holds for a left Hilbert \cs-module. 
\end{definition}

We recall the following well-known result (see~\cite[p.~65]{FGV}), whose proof is included here: 
\begin{lemma}\label{lem: part}
Let $M_\Bs$ be a right Hilbert \cs-module over a unital \cs-algebra $\Bs$. 
Then $M_\Bs$ is full if and only if \,$\spa\{\ip{x}{y}_\Bs \ | \ x,y\in M_\Bs\}=\Bs$.
\end{lemma}
\begin{proof}
If $M_\Bs$ is full, for any $\epsilon >0$, we can find a natural number $n \in \NN_0$ and elements $x_j,y_j\in M$, with $j=1,\dots,n$, such that 
\begin{equation*}
\|\sum_{j=1}^{n} \ip{x_j}{y_j}_\Bs-1_\Bs\|_\Bs<\epsilon. 
\end{equation*}
Taking $\epsilon\leq 1$, we see that $\sum_{j=1}^{n}\ip{x_j}{y_j}_\Bs$ is  invertible i.e.~there exists an element $b_\epsilon$ in $\Bs$ 
such that $(\sum_{j=1}^{n}\ip{x_j}{y_j}_\Bs)b_\epsilon=1_\Bs$.
Hence $\sum_{j=1}^{n}\ip{x_j}{y_jb_\epsilon}_\Bs=1_\Bs$, i.e.~$1_\Bs$ is in the ideal $\spa\{\ip{x}{y}_\Bs\ | \ x,y\in M_\Bs\}$ that therefore coincides with $\Bs$. 
\end{proof}

We note that the notion of Hilbert C*-modules behaves naturally under quotients:

\begin{proposition}\label{prop: quot-mod}
Let $M_\As$ be a right Hilbert \cs-module over a unital \cs-algebra $\As$ and $\Is\subset\As$ an involutive ideal in $\As$. 
Then the set $M\Is:=\{\sum_{j=1}^N x_ja_j \ | \ x_j\in M, \ a_j\in \Is, \ N\in \NN_0\}$ is a submodule of $M$. 
The quotient module $M/(M\Is)$ has a natural structure as a right Hilbert \cs-module over the quotient \cs-algebra $\As/\Is$. 
If $M$ is full over $\As$, also $M/(M\Is)$ is full over $\As/\Is$. 
A similar statement holds for a left Hilbert \cs-module. 
\end{proposition}
\begin{proof}
Clearly $M\Is$ is a submodule of the right $\As$-module $M$. 
It is immediately checked that the operation of right multiplication by elements of $\As/\Is$ and the $\As/\Is$-valued inner product given by:
\begin{gather*}
(x+M\Is)\cdot (a+\Is):=xa+M\Is, \quad \forall x+M\Is\in M/(M\Is) \ \forall a+\Is\in \As/\Is, \\
\ip{x+M\Is}{y+M\Is}_{\As/\Is}:=\ip{x}{y}_\As +\Is, \quad \forall x+M\Is,\,y+M\Is\in M/(M\Is), 
\end{gather*}
are well-defined so that $M/(M\Is)$ becomes a right Hilbert \cs-module over $\As/\Is$. 
Of course if $\ip{M}{M}=\As$, also $\ip{M/(M\Is)}{M/(M\Is)}=\As/\Is$. 
\end{proof}

\begin{definition}
A \emph{morphism of right Hilbert \cs-modules}, from $(M_\Bs,\ip{\cdot}{\cdot}_\Bs)$ into $(N_\Bs,\ip{\cdot}{\cdot}'_\Bs)$ is an adjointable map i.e.~a function $T:M_\Bs\to N_\Bs$ such that
\begin{gather*}
\exists S: N\to M, \quad \ip{S(x)}{y}_\Bs=\ip{x}{T(y)}'_\Bs, \quad \forall x\in N, \ 
\forall y\in M.
\end{gather*}
\end{definition}

\begin{remark}
It is well-known, see e.g.\ N.~Landsman~\cite[Theorem~3.2.5]{L1}, that an adjoint\-able map $T\colon M_\Bs\to N_\Bs$ between Hilbert \cs-modules is necessarily continuous and $\Bs$-linear:
\begin{gather*}
T(xa+yb)=T(x)a+T(y)b, \quad \forall x,y\in M, \ \forall a,b\in \Bs. 
\end{gather*}
Furthermore, the family $\End(M_\Bs)$ of morphisms on $\M_\Bs$ has a natural structure of a unital \cs-algebra. 
\end{remark}
Given $x,y\in M_\Bs$, an operator $\theta_{x,y}:M_\Bs\to M_\Bs$ of the form 
\begin{equation}\label{eq: thetaxy}
\theta_{x,y}:z\mapsto x\cdot\ip{y}{z}_\Bs
\end{equation} 
is clearly a morphism of the right Hilbert \cs-module $M_\Bs$ with adjoint given by $\theta_{y,x}$. 

\begin{definition}\label{rem: comp}
A \emph{finite-rank} operator of the Hilbert \cs-module $M_\Bs$ is a finite linear combination of operators of the form $\theta_{x,y}$,  
$x,y\in M_\Bs$, as described in~\eqref{eq: thetaxy}. 

The family $\Ks(M_\Bs)$ of \emph{compact} operators of the right Hilbert \cs-module $M_\Bs$ is by definition the \cs-subalgebra of $\End(M_\Bs)$ generated by the finite-rank operators. 
\end{definition}
\begin{definition}\label{def: twist}
Let $M_\Bs$ be a right unital module over a unital ring $\Bs$ and let $\alpha:\As\to\Bs$ be a unital homomorphism of rings. The \emph{right twisted module} of $M_\Bs$ by the homomorphism $\alpha$ is the right unital module $M_\alpha$ over the unital ring $\As$ with the right action defined by:
\begin{equation*}
x\cdot a:=x\cdot \alpha(a), \quad \forall x\in M, \ \forall a\in \As.
\end{equation*}
The \emph{left twisted module} of ${}_\Bs M$ by the homomorphism $\alpha:\As\to\Bs$ is analogously defined. 
\end{definition}

\begin{remark}\label{rem: twist}
If $M_\Bs$ is a right (pre-)Hilbert \cs-module and $\alpha:\As\to\Bs$ is an isomorphism of unital \cs-algebras, then the right $\As$-module $M_\alpha$ obtained by right twisting $M_\Bs$ by the isomorphism $\alpha$ has a natural structure as a (pre-)Hilbert \cs-module over $\As$ with the inner product given by $\ip{x}{y}_\As:=\alpha^{-1}(\ip{x}{y}_\Bs)$. 
\end{remark}

\begin{proposition}\label{prop: twist} 
Let $\alpha\colon \As\to\Bs$ be a unital isomorphism of unital rings. Let $M_\As$ and $N_\Bs$ be unital right modules over $\As$ and respectively $\Bs$.   
Then $\Phi\colon M_\As\to N_\alpha$ is a morphism of right modules over $\As$ if and only if $\Phi\colon M_{\alpha^{-1}}\to N_\Bs$ is a morphism of right $\Bs$-modules. 

The result holds true also when $M_\As$ and $N_\Bs$ are (pre-)Hilbert \cs-modules and $\Phi:M_\As\to N_\alpha$ is a morphism of (pre-)Hilbert \cs-modules over $\As$.
\end{proposition}
\begin{proof}
Clearly $\Phi(x\cdot a)=\Phi(x)\cdot \alpha(a)$ if and only if $\Phi(x\cdot \alpha^{-1}(b))=\Phi(x)\cdot b$. 
Also $\Phi:M_\As\to N_\alpha$ is adjointable, with adjoint $\Psi$, if and only $\Phi:M_{\alpha^{-1}}\to N_\Bs$ is adjointable with the same adjoint: $\alpha^{-1}(\ip{x}{\Phi(y)}_\Bs)=\ip{\Psi(x)}{y}_\As$
if and only if $\ip{x}{\Phi(y)}_\Bs=\alpha(\ip{\Psi(x)}{y}_\As)$, for all $x\in N$, 
$y\in M$. 
\end{proof}

\subsection{Hilbert \cs-bimodules and Morita Equivalence}\label{app: morita}

Recall that a unital bimodule ${}_\As M_\Bs$ over two unital rings $\As$ and $\Bs$ is a left unital $\As$-module and a right unital $\Bs$-module such that 
$(a\cdot x)\cdot b=a\cdot(x\cdot b)$, for all $a\in \As$, $b\in \Bs$ and $x\in M$. 

\begin{definition}\label{def: bimod}
A \emph{pre-Hilbert \cs-bimodule} ${}_\As M_\Bs$ over a pair of unital \hbox{\cs-algebras} $\As,\Bs$ is a left pre-Hilbert \cs-module over $\As$ and a right pre-Hilbert \cs-module over $\Bs$ such that:
\begin{gather}
(a\cdot x)\cdot b= a\cdot (x\cdot b) \quad \forall a\in \As, \ x\in M, \ b\in \Bs, \\
\ip{x}{ay}_\Bs=\ip{a^*x}{y}_\Bs \quad \forall x,y\in M, \ \forall a\in \As, \label{eq: c} \\
{}_\As\ip{xb}{y}={}_\As\ip{x}{yb^*} \quad \forall x,y\in M, \ \forall b\in \Bs.
\end{gather}

A \emph{correspondence from $\As$ to $\Bs$} is an $\As$-$\Bs$-bimodule that is also a right Hilbert \cs-module over $\Bs$ whose $\Bs$-valued inner product satisfies property~\eqref{eq: c}. 

A \emph{Hilbert \cs-bimodule} ${}_\As M_\Bs$ is a pre-Hilbert \cs-bimodule over $\As$ and $\Bs$ that is simultaneously a left Hilbert \cs-module over $\As$ and a right Hilbert \cs-module over $\Bs$. 

A Hilbert \cs-bimodule is \emph{full} if it is full as a right and also as a left module. 

A full Hilbert \cs-bimodule over the \cs-algebras $\As$-$\Bs$ is said to be an \emph{imprimitivity bimodule} or an \emph{equivalence bimodule} if:
\begin{gather}
{}_\As\ip{x}{y}\cdot z=x\cdot \ip{y}{z}_\Bs, \quad \forall x,y,z\in M. \label{eq: imp}
\end{gather}
\end{definition}

\begin{remark}
Note that our definitions of pre-Hilbert and Hilbert \cs-bimodule are not necessarily in line with often conflicting similar definitions available in the literature: for example, H.~Figueroa-J.~Gracia-Bondia-J.~Varilly~\cite[Definition~4.7]{FGV} and B.~Abadie-R.~Exel~\cite{AE} require pre-Hilbert \cs-bimodules to satisfy condition~\eqref{eq: imp}; A.~Connes~\cite[Page~159]{C} calls Hilbert \cs-bimodules what we call here correspondences (in this case, only one inner product is assumed). 
In an $\As$-$\Bs$ pre-Hilbert \cs-bimodule there are two, usually different, norms:
\begin{gather*}
{}_M \|x\|:=\sqrt{\|{}_\As\ip{x}{x}\|_\As}, \quad 
\|x\|_M:=\sqrt{\|\ip{x}{x}_\Bs\|_\Bs}, \quad \forall x\in M.  
\end{gather*}
The two norms coincide for an imprimitivity bimodule or, more generally, for a pre-Hilbert \cs-bimodule ${}_\As M_\Bs$ such that ${}_\As\ip{x}{x}x=x\ip{x}{x}_\Bs$, for all $x\in M$. In fact 
\begin{align*}
{}_M\|x\|^4 &=\|{}_\As\ip{x}{x}\|_\As^2=\|{}_\As\ip{x}{x}{}_\As\ip{x}{x}\|_\As =\|{}_\As\ip{x\ip{x}{x}_\Bs}{x}\|_\As \\
&\leq \|\ip{x}{x}_\Bs\|_\Bs\cdot \|{}_\As\ip{x}{x}\|_\As=\|x\|^2_M \cdot {}_M\|x\|^2.
\end{align*}
\end{remark}

\begin{definition} 
A \emph{morphism of correspondences} from $\As$ to $\Bs$ is a morphism of right Hilbert \cs-modules over $\Bs$ that further satisfies:
\begin{equation}
T(ax)=aT(x), \quad \forall x\in M, \ \forall a\in \As. \label{eq: mo-co}
\end{equation}
A \emph{morphism of (pre-)Hilbert \cs-bimodules} is just a morphism of right and left (pre-) Hilbert \cs-bimodules. 
\end{definition}
\begin{remark}
Morphisms of correspondences are just morphisms of bimodules that are adjointable for the right \cs-module structure. 

Note that in a (pre-)Hilbert \cs-bimodule there are in general two different notions of left and of right adjoint of a morphism. 
The left and right adjoints of a morphism coincide if and only if ${}_\As\ip{x}{y}=0_\As\iff\ip{x}{y}_\Bs=0_\Bs$, for all $x,y\in M$.
This condition is true for all full (pre)-Hilbert \cs-bimodules such that
\begin{equation}\label{eq: f-o-b}
{}_\As\ip{x}{y}x=x\ip{y}{x}_\Bs, \quad \forall x,y\in {}_\As M_\Bs.
\end{equation}
\end{remark}

\begin{proposition}\label{prop: imp-a-b}
If ${}_\As M_\Bs$ is an imprimitivity bimodule over the unital \cs-algebras $\As$ and $\Bs$, the map $T: \As\to \Ks(M_\Bs)$ given by $\alpha\mapsto T_\alpha$, where we define  $T_\alpha(x):=\alpha\cdot x$, is an isomorphism of \cs-algebras. Furthermore the \cs-algebra of compact operators $\Ks(M_\Bs)$ coincides with the family of finite-rank operators. 
\end{proposition}
\begin{proof}
Clearly $T_\alpha$ is a morphism of the Hilbert \cs-module $M_\Bs$ with adjoint given by $T_{\alpha^*}$. 
The map $\alpha\mapsto T_\alpha$ is a unital involutive homomorphism from $\As$ to $\End(M_\Bs)$ and so its image is a unital \cs-subalgebra of the \cs-algebra $\End(M_\Bs)$. Furthermore, from the fullness of $M_\Bs$, we see that 
$\alpha\mapsto T_\alpha$ is injective so that $\As$ is isomorphic to its image under $T$ in $\End(M_\Bs)$. 

The image of $T$ contains all the finite-rank operators, for if $S=\sum_k\theta_{x_k,y_k}$, with $x_k,y_k\in M_\Bs$, then for all $z\in M_\Bs$,  
\begin{align*}
S(z) &= \sum_k\theta_{x_k,y_k}(z) 
=\sum_k x_k\ip{y_k}{z}_\Bs =\sum_k{}_\As\ip{x_k}{y_k}z
=T_{\alpha}(z), 
\end{align*}
where $\alpha:=\sum_k{}_\As\ip{x_k}{y_k}$. 
Since, by lemma~\ref{lem: part}, every $\alpha\in \As$ can always be written as a finite combination $\alpha=\sum_k{}_\As\ip{x_k}{y_k}$, we see that $T_\alpha$ is always a finite-rank operator, and hence the image of $T$ coincides with the family of finite-rank operators. 

Since the closure of the finite-rank operators is the \cs-algebra of compact operators $\Ks(M_\Bs)$, we see that $T$ is an isomorphism of \cs-algebras from $\As$ onto $\Ks(M_\Bs)$ and that $\Ks(M_\Bs)$ coincides with the family of finite-rank operators.  
\end{proof}

There is a natural notion of \emph{Rieffel interior tensor product} between Hilbert \cs-modules and correspondences~\cite{R1}:
\begin{proposition}
Given two unital \cs-algebras $\As,\Bs$, let $M_\As$ be a right Hilbert \cs-module over $\As$ and let ${}_\As N_\Bs$ be a correspondence from $\As$ to $\Bs$. 
The algebraic tensor product $M\otimes_\As N$ of the right $\As$-module $M$ with the $\As$-$\Bs$-bimodule $N$ is naturally a right Hilbert \cs-module over $\Bs$ with the unique $\Bs$-valued inner product such that:
\begin{equation*}
\ip{x_1\otimes y_1}{x_2\otimes y_2}_\Bs=\ip{y_1}{\ip{x_1}{x_2}_\As\cdot y_2}_\Bs, \quad \forall x_1,x_2\in M, \ \forall y_1,y_2\in N.
\end{equation*}
Similarly, the algebraic tensor product $M\otimes_\Bs N$, of a pair of (pre-)Hilbert \cs-bimodules ${}_\As M_\Bs$, ${}_\Bs N_\Cs$ has a natural structure of (pre-)Hilbert \cs-bimodule on the unital \cs-algebras $\As$-$\Cs$ where the ``left-action'' of $\As$ satisfies:
\begin{equation*}
a (x\otimes y):=(ax)\otimes y, \quad \forall a\in\As, \ \forall x\in M, \ y\in N. 
\end{equation*} 
\end{proposition}
There is also a natural notion of \emph{Rieffel dual} of a (pre-)Hilbert \cs-bimodule~\cite{R1} that is uniquely defined (up to isomorphism) via the following proposition: 
\begin{proposition}\label{def: ri-dual}
Let ${}_\Bs M_\As$ be a (pre-)Hilbert \cs-bimodule. Then there exist a (pre-) Hilbert \cs-bimodule ${}_\As M^*_\Bs$ and an anti-homomorphism of bimodules 
$\iota: {}_\Bs M_\As \to {}_\As M^*_\Bs$, i.e.~a map such that
$\iota(bxa)=a^*\iota(x)b^*$  $\forall x\in M$ $\forall a\in \As$ $\forall b\in \Bs$, 
satisfying the following universal property:
for every (pre-)Hilbert \cs-bimodule ${}_\As N_\Bs$ and any anti-homomorphism of bimodules 
$\Phi: {}_\Bs M_\As \to {}_\As N_\Bs$ 
there exists a unique homomorphism of bimodules \hbox{$\Phi':{}_\As M^*_\Bs \to {}_\As N_\Bs$} such that $\Phi=\Phi'\circ\iota$.
\end{proposition}
\begin{proof}
We take $M^*:=M$ as sets, but we define on $M^*$ the following bimodule structure:
\begin{gather*}
a\cdot x:=x a^*, \quad \forall x\in M^*=M, \quad \forall a\in \As,\\
x\cdot b:=b^* x, \quad \forall x\in M^*=M, \quad \forall b\in \Bs. 
\end{gather*}
It is easily checked that ${}_\As M^*_\Bs$ is a bimodule and that it becomes a (pre-)Hilbert \cs-bimodule if the inner products on $M^*$ are defined as follows:
\begin{gather*}
\ip{x}{y}'_\Bs:={}_\Bs\ip{x}{y}^*, \quad \forall x,y\in M^*, \\
{}_\As\ip{x}{y}':=\ip{x}{y}^*_\As, \quad \forall x,y\in M^*,
\end{gather*} 
where ${}_\As\ip{x}{y}'$ and $\ip{x}{y}'_\Bs$ denote the inner products on ${}_\As M^*_\Bs$. 

Clearly the identity map $\iota:M\to M^*$ is an anti-homomorphism of bimodules and for any anti-homomorphism of bimodules $\Phi:{}_\Bs M_\As \to {}_\As N_\Bs$, $\Phi':=\Phi$ is the unique homomorphism of bimodules $\Phi': {}_\As M^*_\Bs \to {}_\As N_\Bs$ such that $\Phi=\Phi'\circ \iota$. 
\end{proof}
The pair $(\iota,{}_\As M^*_\Bs)$ is unique up to isomorphism (as for any concept defined through a universal property) and is called the dual of the \hbox{(pre-)Hilbert} \cs-bimodule ${}_\Bs M_\As$.

\begin{definition}\label{def: m-p}
The \emph{Morita category} is the involutive category\footnote{By an \emph{involutive category} we mean a category $\Cf$ equipped with an involutive contravariant endofunctor  acting identically on the objects of $\Cf$ i.e.~a map $*:\Cf\to\Cf$ such that $(x^*)^*=x$ and $(x\circ y)^*=y^*\circ x^*$ for all $x,y\in \Cf$.} with objects the unital associative  rings, with  morphisms the isomorphism classes of bimodules, with composition the isomorphism classes of the tensor product of bimodules, and with involution given by isomorphism classes of the dual bimodules. 
The \emph{(algebraic) Picard groupoid} is the nerve of the Morita category\footnote{The nerve of a category is its class of invertible arrows.}. 
Two unital associative rings are \emph{Morita equivalent} if they are in the same orbit of the Picard groupoid. 
\end{definition}

Here we are interested only in the full subcategory of the Morita category whose objects are unital \cs-algebras. In this case, it is usually better to ``restrict'' also the family of allowed arrows as long as the new category preserves the notion of Morita equivalence i.e.~its nerve has the same orbits of the Picard groupoid.\footnote{There are also  interesting versions of Morita theory for involutive unital algebras (see P.~Ara~\cite{A} and H.~Bursztyn-S.~Waldmann~\cite{BW}).} 

The category described in the following definition is the \emph{Morita-Rieffel category} of unital \cs-algebras and it plays a key role in the discussion of the horizontal categorification of Gel'fand Theorem~\cite{BCL2}. 
\begin{definition}\label{def: mr-pr}
The \emph{Morita-Rieffel} category is the subcategory of the Morita category whose objects are unital \cs-algebras, whose arrows are the isomorphism classes of correspondences and whose composition is the Rieffel tensor product of correspondences.  
The nerve of this category is the (algebraic) \emph{Picard-Rieffel} groupoid. 
Two \cs-algebras in the the same orbit of the Picard-Rieffel groupoid are said to be 
\emph{strongly Morita equivalent}~\cite{Rie}. 
\end{definition} 
\begin{remark}
Note that the Morita-Rieffel category is not an involutive category (the substitution of bimodules with correspondences ``breaks the symmetry'' between left and right module structures). It is possible to eliminate this problem considering other subcategories of the Morita category. 
Two possible natural choices are the involutive subcategory of the Morita category  consisting of isomorphism classes of (pre-)Hilbert \cs-bimodules or (whenever it is necessary to have a unique Banach norm and a unique notion of adjoint of a morphism of the bimodules involved) the subcategory consisting of full Hilbert \cs-bimodules such that property~\eqref{eq: f-o-b} is satisfied. 
In these cases the involution is given by the Rieffel dual of the bimodules. 
\end{remark}

The following proposition is a well-known result (see e.g.~\cite[Section~8.8]{GMS} for a review).

\begin{proposition}
Two unital \cs-algebras $\As$ and $\Bs$ are Morita equivalent if and only if there exists an imprimitivity bimodule ${}_\As M_\Bs$. The Picard-Rieffel groupoid consists of isomorphism classes of imprimitivity Hilbert \cs-bimodules. Moreover,
the notions of Morita equivalence and strong Morita equivalence coincide. 
\end{proposition}
\begin{proof} 
If $\As$ and $\Bs$ are Morita equivalent, there exists bimodules ${}_\As M_\Bs$ and 
${}_\Bs N_\As$ such that $M\otimes_\Bs N \simeq \As$ and $N\otimes_\As M\simeq \Bs$. 
Any bimodule ${}_\As M_\Bs$ with the previous properties is necessarily finite projective~\cite[Theorem~10.4.3]{GMS}. 
Any finite projective right module can be equipped with an inner product that makes it a correspondence from $\As$ to $\Bs$ and hence ${}_\As M_\Bs$ must be an imprimitivity bimodule. 
\end{proof}
\subsection{Imprimitivity Bimodules on Abelian \cs-algebras.}\label{sec: imp-bim}

It is well-known that in some cases imprimitivity bimodules can be used to construct explicit isomorphisms between the associated 
\cs-algebras, see e.g.~\cite[Lemma~10.19]{Bo}. 
In this subsection we follow a similar route, recovering and further elaborating on a ``classical'' result~\cite[Theorem 3.1 and Corollary 3.3]{Rie3} that is certainly folklore among specialists. 
For the sake of self-containment we present a full account of the situation at hand.

\medskip

The following theorem is motivated by P.~Ara~\cite[Theorem~4.2]{A}. 

\begin{theorem}\label{lemma: phiM}
Let ${}_\As M_\Bs$ be an $\As$-$\Bs$ imprimitivity bimodule, where $\As$ and $\Bs$ are commutative unital \cs-algebras. 
Then there exists a unique canonical isomorphism $\phi_M:\As\to \Bs$ such that: 
\begin{equation} \label{eq: isomorphism}
\phi_M({}_\As\ip{x}{y})=\ip{y}{x}_\Bs, \quad \forall x,y\in M. 
\end{equation}
Moreover the canonical isomorphism $\phi_M$ satisfies the following property:
\begin{equation}
a\cdot x=x\cdot \phi_M(a), \quad \forall x\in M, \ \forall a\in \As. 
\end{equation}
\end{theorem}
\begin{proof}
The uniqueness of the map follows from the fullness of the left Hilbert C*-module ${}_\As M$.
By the fullness of the right Hilbert C*-module $M_\Bs$ we can write $1_\Bs$ as a finite sum 
$1_\Bs=\sum_{j=1}^n \ip{w_j}{z_j}_\Bs,$ where $w_j,z_j\in M$, $j = 1,\dots,n$. 
For any $a \in \As$, define
\begin{equation}\label{eq: phi-def}
   \phi_M(a) = \sum_{j=1}^n \ip{w_j}{a z_j}_\Bs, 
\end{equation}
where $w_j,z_j\in M$ are such that $\sum_{j=1}^n \ip{w_j}{z_j}_\Bs=1_\Bs$. 

To show that $\phi_M$ is well-defined, let $w_j, z_j$ and $x_k, y_k$ be two pairs of finite sequences such that 
$\sum_j\ip{w_j}{z_j}_\Bs=1_\Bs$ and $\sum_k\ip{x_k}{y_k}_\Bs=1_\Bs$.
Write $b=\sum_j\ip{w_j}{a z_j}_\Bs$.  Then
\begin{align*}
\ip{x_k}{y_k}\!_\Bs\, b & 
	= \ip{x_k}{y_k}_\Bs \sum_j\ip{w_j}{az_j}_\Bs  \\
& = \sum_{j}\ip{x_k}{y_k\ip{w_j}{az_j}_\Bs}_\Bs = \sum_{j}\ip{x_k}{{}_\As\ip{y_k}{w_j}\,az_j}_\Bs \\
& = \sum_{j}\ip{x_k}{a\,{}_\As\!\ip{y_k}{w_j}z_j}_\Bs = \sum_{j}\ip{x_k}{ay_k\ip{w_j}{z_j}_\Bs}_\Bs \\
& = \ip{x_k}{ay_k}_\Bs.
\end{align*}
It follows that $b = \sum_k \ip{x_k}{ay_k}_\Bs$, which shows that $\phi_M(a)$ is well-defined.

We now show that $\phi_M$ is a homomorphism of algebras. Clearly $\phi_M$ is additive and $\CC$-linear. The multiplicativity follows from:
\begin{align*}
\phi_M(a)&\cdot \phi_M(a') = \sum_j\ip{w_j}{az_j}_\Bs \sum_k\ip{w'_k}{a'z'_k}_\Bs \\
	& = \sum_{j,k} \ip{w_j}{az_j\ip{w'_k}{a'z'_k}_\Bs}_\Bs 
			= \sum_{j,k} \ip{w_j}{a{}_\As\ip{z_j}{w'_k}a'z'_k}_\Bs \\
	& = \sum_{j,k} \ip{w_j}{{}_\As\!\ip{z_j}{w'_k}aa'z'_k}_\Bs 
			= \sum_{j,k} \ip{w_j}{z_j\ip{w'_k}{aa'z'_k}_\Bs}_\Bs \\
	& = \sum_{j,k} \ip{w_j}{z_j}_\Bs\ip{w'_k}{aa'z'_k}_\Bs 
			= \sum_{k} \ip{w'_k}{aa'z'_k}_\Bs = \phi_M(aa'). 
\end{align*}
Of course $\phi_M$ is unital: $\phi_M(1_\As)=\sum_j\ip{w_j}{1_\As z_j}_\Bs =\sum_j\ip{w_j}{z_j}_\Bs=1_\Bs$.
To prove the involutivity of $\phi_M$, note that if $\sum_j\ip{w_j}{z_j}_\Bs=1_\Bs$, taking the adjoints, we also have  $\sum_j\ip{z_j}{w_j}_\Bs=1_\Bs$. Hence 
\begin{equation*}
\phi_M(a^*)	= \sum_j\ip{w_j}{a^*z_j}_\Bs = \sum_j\ip{aw_j}{z_j}_\Bs 
						 = \sum_j\ip{z_j}{aw_j}^*_\Bs  = \phi_M(a)^*.
\end{equation*}

Similarly, there is a canonical homomorphism $\psi_M: \Bs\to \As$ defined by: 
\begin{equation*}
\psi_M(b):=\sum_i{}_\As\ip{t_ib}{u_i} \quad \forall b\in \Bs,
\end{equation*}
where $t_i,u_i\in M$ is a pair of finite sequences such that $\sum_i{}_\As\ip{t_i}{u_i}=1_\As$. Then
\begin{align*}
\psi_M(\phi_M(a)) & = \sum_i {}_\As\ip{t_i\phi_M(a)}{u_i} \\
									& = \sum_{i,j} {}_\As\ip{t_i\ip{w_j}{az_j}_\Bs}{u_i} 
										= \sum_{i,j} {}_\As\ip{{}_\As\ip{t_i}{w_j}az_j}{u_i} \\
									& = \sum_{i,j} a{}_\As\ip{t_i\ip{w_j}{z_j}_\Bs}{u_i} 
										= \sum_{i} a{}_\As\ip{t_i}{u_i} = a. 
\end{align*}
By the same argument, we can show that $\phi_M(\psi_M(b))=b$ for all $b\in \Bs$. Hence $\psi_M$ is the inverse of $\phi_M$, 
which implies that $\phi_M$ is an isomorphism.

To establish \eqref{eq: isomorphism}, let $w_j,z_j\in M$ be finite sequences such that $\sum_j\ip{w_j}{z_j}_\Bs=1_\Bs$. 
Define $\alpha:=\sum_j{}_\As\ip{z_j}{w_j}$ and note that 
\begin{equation} \label{eq: defining alpha}
\phi_M({}_\As\ip{x}{y})=\ip{y}{\alpha x}_\Bs, \quad \forall x,y\in M,
\end{equation}
which follows from this computation:
\begin{align*}
\phi_M({}_\As\ip{x}{y}) &= \sum_j\ip{w_j}{{}_\As\ip{x}{y}z_j}_\Bs
		= \sum_j\ip{w_j}{x\ip{y}{z_j}_\Bs}_\Bs \\
&= \sum_j\ip{w_j}{x}_\Bs \ip{y}{z_j}_\Bs 
		= \sum_j \ip{y}{z_j}_\Bs \ip{w_j}{x}_\Bs  \\
&= \sum_j \ip{y}{z_j\ip{w_j}{x}_\Bs}_\Bs 
		=\sum_j \ip{y}{{}_\As\ip{z_j}{w_j}x}_\Bs \\
&=  \ip{y}{\sum_j{}_\As\ip{z_j}{w_j}x}_\Bs 
		= \ip{y}{\alpha x}_\Bs.
\end{align*}
The element $\alpha\in \As$ is independent from the choice of the finite sequences $w_j,z_j\in M$ such that $\sum_j\ip{w_j}{z_j}_\Bs=1_\Bs$. In fact, given another pair of finite sequences $w_i',z_i'\in M$ such that $\sum_i\ip{w_i'}{z_i'}_\Bs=1_\Bs$, we see that 
$\phi_M({}_\As\ip{x}{y})=\ip{y}{\alpha'x}_\Bs$, where $\alpha':=\sum_i{}_\As\ip{z_i'}{w_i'}$ so that $\ip{y}{\alpha x}_\Bs=\ip{y}{\alpha' x}_\Bs$ for all $x,y\in M$ that implies immediately $(\alpha-\alpha')x=0_M$ that (by the fulless of the module ${}_\As M$) implies $\alpha'=\alpha$. 

We see that $\alpha$ is Hermitian because for all $x,y\in M$:
\begin{align*}
\ip{x}{\alpha y}_\Bs &= \phi_M({}_\As\ip{y}{x}) = \phi_M({}_\As\ip{x}{y}^*) \\
		&= \phi_M({}_\As\ip{x}{y})^* = \ip{y}{\alpha x}_\Bs^*=\ip{\alpha x}{y}_\Bs 
		= \ip{x}{\alpha^*y}_\Bs, 
\end{align*}
which implies that $\alpha=\alpha^*$. 

We can actually prove that $\alpha\in\As$ is positive. 
Since $\phi_M:\As\to \Bs$ is an isomorphism, the map 
$(x,y)\mapsto \phi_M({}_\As\ip{x}{y})=\ip{y}{\alpha x}_\Bs$ is a $\Bs$-valued inner product on $M$. Hence $\phi_M({}_\As\ip{x}{x})=\ip{x}{\alpha x}_\Bs$ is a positive element in $\Bs$ for all $x\in M$. 
Considering the positive and negative parts of the Hermitian element $\alpha$, i.e.~the unique pair of positive elements $\alpha_+,\alpha_-\in \As_+$ such that $\alpha=\alpha_+-\alpha_-$ with $\alpha_+\alpha_-=0_\As$, we see that 
\begin{equation*}
\ip{x}{\alpha_+ x}_\Bs -\ip{x}{\alpha_- x}_\Bs \in \Bs_+, \quad \forall x\in M.
\end{equation*}
From the calculation below,  
\begin{align*}
\ip{x}{\alpha_+x}_\Bs\ip{x}{\alpha_-x}_\Bs &= \ip{x}{\alpha_+x\ip{x}{\alpha_-x}_\Bs}_\Bs \\
&= \ip{x}{\alpha_+{}_\As\ip{x}{x}\alpha_-x}_\Bs= \ip{x}{\alpha_+\alpha_-{}_\As\ip{x}{x}x}_\Bs \\
&=\ip{x}{0_\As{}_\As\ip{x}{x}x}_\Bs= 0_\Bs, 
\end{align*}
it follows that the positive terms 
$\ip{x}{\alpha_\pm x}_\Bs=\ip{\alpha_\pm^{1/2}x}{\alpha_\pm^{1/2}x}_\Bs$ are the positive and negative parts of the positive element $\ip{x}{\alpha x}_\Bs$. 
Therefore $\ip{x}{\alpha_-x}_\Bs=0_\Bs$ for all $x\in M$, and thus 
$\alpha_-=0_\As$, and so $\alpha$ is positive. 

Next we prove that $\|\alpha\|_\As\leq1$. 
Consider the operator $T_\alpha: M_\Bs\to M_\Bs$ given by 
\begin{equation*}
T_\alpha (x):=\alpha\cdot x, \quad \forall x\in M  
\end{equation*} 
and note that $\|T_\alpha\|\leq 1$ because, for all $x\in M$, 
\begin{align*} 
\|T_\alpha(x)\|^2 &= \|\ip{T_\alpha(x)}{T_\alpha (x)}_\Bs\| = 
\|\ip{T_\alpha(x)}{\alpha x}_\Bs\|=\|\phi_M({}_\As\ip{T_\alpha(x)}{x})\|  \\
&= \|{}_\As\ip{T_\alpha(x)}{x}\|\leq\|T_\alpha(x)\|\cdot\|x\|.  
\end{align*} 
By proposition~\ref{prop: imp-a-b}, the map $T\colon \As\to \Ks(M_\Bs)$, $\alpha \mapsto T_\alpha$, is an isomorphism from $\As$ onto 
the \cs-algebra of compact operators $\Ks(M_\Bs)$. Thus
\begin{equation*}
\|\alpha\|=\|T_\alpha\|\leq 1, \quad \forall\alpha\in \As.
\end{equation*}

In a completely similar way, we can find a positive Hermitian element $\beta \in \Bs$ such that $\|\beta\| \le 1$ and that
\begin{equation} \label{eq: defining beta}
\psi_M(\ip{x}{y}_\Bs)={}_\As\ip{y\beta}{x}, \quad \forall x,y\in M. 
\end{equation}

The two elements $\alpha$ and $\beta$ are related by
$\phi_M(\alpha)\beta=1_\Bs$ and $\psi_M(\beta)\alpha=1_\As.$
In order to prove this, we first note that 
\begin{equation}\label{eq: ax}
x\cdot \phi_M(a) = a \cdot x, \quad \forall x\in M, \quad \forall a\in \As.
\end{equation}
In fact, if $w_j,z_j\in M$ is a pair of sequences such that $\sum_j\ip{w_j}{z_j}_\Bs=1_\Bs$, equation~\eqref{eq: ax} follows from this direct computation:
\begin{align*}
x\cdot\phi_M(a)&=x \sum_j\ip{w_j}{a z_j}_\Bs=\sum_j{}_\As\ip{x}{w_j}\, a z_j \\ 	
					&=\sum_ja\, {}_\As\!\ip{x}{w_j} z_j = \sum_j a x \ip{w_j}{z_j}_\Bs = a\cdot x.
\end{align*}
Next we see that 
\begin{equation}\label{eq: axb}
\alpha\cdot x\cdot \beta=x, \quad \forall x\in M. 
\end{equation}
To see this, we apply \eqref{eq: defining alpha} and \eqref{eq: defining beta} to the following calculation:
\begin{align*}
  \ip{\alpha\cdot x\cdot \beta}{y}_\Bs = \ip{x\cdot \beta}{\alpha\cdot y}_\Bs &= \phi_M({}_\As\ip{y}{x \cdot \beta}) 
  =  \phi_M({}_\As\ip{y \cdot \beta}{x}) \\
  &= \phi_M(\psi_M(\ip{x}{y}_\Bs) = \ip{x}{y}_\Bs.
\end{align*}
From~\eqref{eq: ax} and~\eqref{eq: axb}, we obtain $x\phi_M(\alpha)\beta=x$ for all $x \in M$, which 
implies $\phi_M(\alpha)\beta=1_\Bs$, by the fullness of the module $M_\Bs$. Similarly, we have $\psi_M(\beta)\alpha=1_\As$.

It follows that $\alpha$ and $\beta$ are invertible and $\|\alpha^{-1}\|=\|\psi_M(\beta)\|=\|\beta\|\leq 1$.  
Since $\alpha$ and $\alpha^{-1}$ are positive elements with norm no larger than one in the commutative \cs-algebra $\As$, 
we have $\alpha=1_\As$.
\end{proof}

\begin{definition}
Let ${}_\As M$ be a left module over an algebra $\As$ and denote by $\As^\circ$ the opposite algebra\footnote{Recall that the opposite algebra $\As^\circ$ of an algebra $\As$ is just the vector space $\As$ equipped with the multiplication $a\cdot_{\As^\circ} b:=b\cdot_\As a$.} 
of $\As$. The \emph{right symmetrized bimodule} of ${}_\As M$ is the $\As$-$\As^\circ$ bimodule ${}_\As M^s_{\As^\circ}$ with right multiplication defined by:
\begin{equation*}
x\cdot a :=ax, \quad \forall x\in M, \ \forall a\in \As. 
\end{equation*}
In a similar way, given a right module $M_\As$, we define its 
\emph{left symmetrized bimodule} ${}_{\As^\circ}\!\!\!{}^{\ s} M_\As$ via the left multiplication given by $a\cdot x:=xa$ for all $x\in M$ and $a\in \As$. 
\end{definition}
In the case of a commutative algebra $\As$, the opposite algebra $\As^\circ$ coincides with $\As$ and the left (respectively right) symmetrized of a module is clearly a symmetric bimodule over $\As$. 
\begin{proposition}
Suppose that ${}_\As M_\Bs$ is an imprimitivity $\As$-$\Bs$-bimodule over two unital commutative \cs-algebras $\As$ and $\Bs$. Let $\phi_M: \As\to \Bs$ be the canonical isomorphism defined in theorem~\ref{lemma: phiM}. 

The bimodule ${}_\As M_{\phi_M}$ coincides with the right symmetrized bimodule 
${}_\As M^s\!\!{}_\As$. 

The bimodule ${}_{\phi_M^{-1}} M_\Bs$ coincides with the left symmetrized bimodule 
${}_\Bs\!{}^s\! M_\Bs$. 
\end{proposition}
\begin{proof}
Take $x\in M$ and $a\in \As$. We already proved in~\eqref{eq: ax} that 
$x\cdot \phi_M(a) = a \cdot x$, for all $x\in M$ and for all $a\in \As$. 

The second part of the proposition $x\cdot b=\phi_M^{-1}(b)\cdot x$ is completed with an exactly similar argument. 

In order to complete the proof, we have to show that the inner products on the right $\phi_M$-twisted bimodule ${}_\As M_{\phi_M}$ coincides with the inner products of the right symmetrized bimodule ${}_\As M^s\!\!{}_\As$ and this is precisely equation \eqref{eq: isomorphism}. 

A similar argument applies to the case of the left symmetrized bimodule ${}_\Bs\!{}^s\! M_\Bs$ and the left $\phi_M$-twisted bimodule ${}_{\psi_M} M_\Bs$. 
\end{proof}

The imprimitivity condition also behaves naturally under quotients.

\begin{proposition}\label{prop: quot-impri}
Let ${}_\As M_\Bs$ be an imprimitivity bimodule over the unital \cs-algebras $\As$ and $\Bs$. Let $\Is$ be an involutive ideal in the \cs-algebra 
$\As$. 
Then $M/(\Is M)$ is an imprimitivity bimodule over $\As/\Is$ and $\Bs/\phi_M(\Is)$. 
\end{proposition}
\begin{proof}
Since $\phi_M:\As\to\Bs$ is an isomorphism of \cs-algebras, if $\Is$ is an involutive ideal in $\As$, also $\phi_M(\Is)\subset\Bs$ is an involutive ideal in $\Bs$. 
Note that property~\eqref{eq: axb} implies that $\Is M=M \phi_M(\Is)$ and so, by proposition~\ref{prop: quot-mod}, $M/(\Is M)=M/(M\phi_M(\Is))$ is a full left Hilbert \cs-module over $\As/\Is$ and a full right Hilbert \cs-module over $\Bs/\phi_M(\Is)$. 
Finally, by direct computation, we have: 
\begin{align*}
{}_{\As/\Is}\ip{x+\Is M}{y+\Is M}(x+\Is M)
&=({}_\As\ip{x}{y}+\Is)(z+\Is M) \\
&={}_\As\ip{x}{y} z + \Is M \\ 
&=x\ip{y}{z}_\Bs+\Is M \\
&=(x+\Is M) (\ip{y}{z}_\Bs+\Is) \\
&=(x+\Is M) \ip{y+ \Is M}{z +\Is M}_{\Bs/\phi_M(\Is)}. 
\end{align*}
\end{proof}

\subsection{Imprimitivity Bimodules in Commutative C*-categories.}\label{sec: imp-bim-cat}

Following P.~Ghez-R.~Lima-J.~Roberts~\cite{GLR} and P.~Mitchener~\cite{M1} we recall the following basic definition. 
\begin{definition}
A \emph{C*-category} is a category $\Cf$ such that: for all $A, B \in \Ob_\Cf$,
the sets $\Cf_{AB}:=\Hom_\Cf(B,A)$ are complex Banach spaces; the compositions are bilinear maps such that $\|xy\|\leq\|x\|\cdot\|y\|$  $\forall x\in \Cf_{AB} \ \forall y\in \Cf_{BC}$; there is an involutive antilinear contravariant functor $*:\Hom_\Cf\to\Hom_\Cf$, acting identically on the objects, such that $\|x^*x\|=\|x\|^2 \ \forall x\in \Cf_{BA}$ and such that $x^*x$ is a positive element in the C*-algebra $\Cf_{AA}$, for every $x\in \Cf_{BA}$ (i.e.~$x^* x = y^* y$ for some $y \in \Cf_{AA}$).
\end{definition}
Every C*-algebra can be seen as a C*-category with only one object. 

\medskip

In a C*-category $\Cf$, the ``diagonal blocks'' $\Cf_{AA}$ are unital C*-algebras and the ``off-diagonal blocks'' $\Cf_{AB}$ are unital Hilbert 
C*-bimodules on the C*-algebras $\As := \Cf_{AA}$ and $\Bs := \Cf_{BB}$. 
For short, we often write ${}_\As\Cf_\Bs:={}_{\Cf_{AA}}{\Cf_{AB}}_{\Cf_{BB}}$ when we want to consider $\Cf_{AB}$ as a bimodule. 

\smallskip 

We say that $\Cf$ is \emph{full} if all the bimodules $\Cf_{AB}$ are imprimitivity bimodules. 
Clearly~\cite[Remark~7.10]{GLR} in a full \cs-category, for all $A,B\in \Ob_\Cf$, $\As:=\Cf_{AA}$ and $\Bs:=\Cf_{BB}$ are always Morita-Rieffel equivalent \cs-algebras with the imprimitivity bimodule ${}_\As\Cf_\Bs$ as an equivalence bimodule. 

\medskip

\begin{lemma}\label{lemma: full}
A \cs-category $\Cf$ is full if and only if it satisfies the following property
\begin{equation}\label{eq: total-full}
\cj{\Cf_{AB}\circ \Cf_{BC}}= \Cf_{AC}, \quad \forall A,B,C \in \Ob_\Cf.
\end{equation}
\end{lemma}
\begin{proof}
Clearly property~\eqref{eq: total-full} is stronger than fullness. 

The fullness of $\Cf$ tells us that $\Cf_{AA}=\cj{\Cf_{AB}\circ\Cf_{BA}}$. 
The continuity of composition implies 
$\cj{\Cf_{AB}\circ \Cf_{BA}}\circ \Cf_{AC}\subset \cj{\Cf_{AB}\circ \Cf_{BA}\circ\Cf_{AC}}$. 
From the following computation 
\begin{align*}
\Cf_{AC}&= 
\Cf_{AA}\circ\Cf_{AC}= 
\cj{\Cf_{AB}\circ \Cf_{BA}}\circ \Cf_{AC} \\ 
&\subset\cj{\Cf_{AB}\circ \Cf_{BA}\circ \Cf_{AC}} \subset 
\cj{\Cf_{AB}\circ \Cf_{BC}}\subset \cj{\Cf_{AC}}=\Cf_{AC}
\end{align*}
we obtain $\Cf_{AC}=\cj{\Cf_{AB}\circ\Cf_{BC}}$.
\end{proof}

We use the previous lemma to show that in a full \cs-category the maps
\begin{equation*}
{}_\As\Cf_\Bs \otimes {}_\Bs\Cf_\Cs \to {}_\As\Cf_\Cs, \quad \text{given by} \quad 
x \otimes y \mapsto x\circ y
\end{equation*}
are isomorphisms of $\As$-$\Cs$-bimodules, 
for all $A,B,C\in \Ob_\Cf$.

\begin{proposition}
If $\Cf$ is a full \cs-category, for all $A,B,C\in \Ob_\Cf$, 
$({}_\As\Cf_\Cs, \circ)$ is a Rieffel interior tensor product for the pair of bimodules ${}_\As\Cf_\Bs$ and ${}_\Bs \Cf_\Cs$.
\end{proposition}
\begin{proof}
We show that there exists an isomorphism 
$T: {}_\As\Cf_\Bs\otimes {}_\Bs\Cf_\Cs\to {}_\As \Cf_\Cs$ 
of Hilbert \hbox{\cs-bimodules} such that $T(x\otimes y)=x\circ y$ for all 
$x\in {}_\As\Cf_\Bs$ and for all $y\in {}_\Bs\Cf_\Cs$.

Consider the composition map $\circ : {}_\As\Cf_\Bs\times{}_\Bs\Cf_\Cs\to {}_\As\Cf_\Cs$ and note that it is a bilinear map of Hilbert \cs-bimodules and hence (by the universal factorization property for tensor products of Hilbert \cs-bimodules) there exists one and only one Hilbert \cs-bimodule morphism 
$T:{}_\As\Cf_\Bs\otimes {}_\Bs\Cf_\Cs\to {}_\As \Cf_\Cs$ such that $T(x\otimes y)=x\circ y$. 

Now we show that, under the fullness condition, the map $T$ is an isomorphism.  

First of all we note that $T$ is an isometric map on the dense sub-bimodule generated by simple tensors:
\begin{align*}
\ip{T(\sum_j x_j\otimes y_j)}{&T(\sum_k x_k\otimes y_k)}_\Cs 
=\sum_{j,k}\ip{x_j\circ y_j}{x_k\circ y_k}_\Cs  \\
&=\sum_{j,k} (x_j\circ y_j)^*\circ (x_k\circ y_k) =
	\sum_{j,k} y_j^*\circ x_j^*\circ x_k\circ y_k \\
&=\sum_{j,k} \ip{y_j}{\ip{x_j}{x_k}_\Bs y_k}_\Cs = 
\sum_{j,k}\ip{x_j\otimes y_j}{x_k\otimes y_k}_\Cs \\
&=\ip{\sum_j x_j\otimes y_j}{\sum_k x_k\otimes y_k}_\Cs. 
\end{align*}
By continuity $T$ extends to an isometry on all of ${}_\As\Cf_\Bs\otimes {}_\Bs\Cf_\Cs$. 
Finally $T$ is surjective because it is an isometry that, from lemma~\ref{lemma: full}, has a dense image in ${}_\As \Cf_\Cs$. 
\end{proof}

Apart from a strictly associative (tensor) product (with partial identities given by ${}_\As\Cf_\As$), the family of imprimitivity bimodules of a full 
\cs-category $\Cf$ is naturally equipped with a strictly antimultiplicative notion of involution given by Rieffel duality (see definition~\ref{def: ri-dual}). 

\begin{proposition}
If $\Cf$ is a full \cs-category, $({}_\Bs\Cf_\As,*)$ is a Rieffel dual of the bimodule ${}_\As\Cf_\Bs$, for all $A,B\in \Ob_\Cf$.
\end{proposition}
\begin{proof} 
Note that the map $*:{}_\As\Cf_\Bs\to{}_\Bs\Cf_\As$ is conjugate-linear, it is an anti-isomorphism of Hilbert \cs-bimodules\footnote{Recall that by an anti-homomorphism $\Phi:{}_\As M_\Bs\to {}_\Bs M_\As$ between unital Hilbert \cs-bimodules $M,N$, we mean a conjugate-linear map that satisfies $\Phi(axb)=b^*\Phi(x)a^*$ for all $x\in M$, $a\in \As$, $b\in \Bs$.} and it is isometric. We need to prove that $({}_\Bs\Cf_\As,*)$ satisfies the universal factorization property for conjugate-linear anti-homomorphisms of bimodules. 

Clearly every conjugate-linear map $\Phi:{}_\As\Cf_\Bs\to {}_\Bs M_\As$, with values in a Hilbert \cs-bimodule ${}_\Bs M_\As$, such that 
$\Phi(axb)=b^*\Phi(x)a^*$ for all 
$x\in M$, $a\in \As$, $b\in \Bs$, factorizes as $\Phi=(\Phi\circ *)\circ *$ via a unique morphism $\Phi\circ *: {}_\Bs\Cf_\As\to {}_\Bs M_\As$ of 
$\Bs$-$\As$-bimodules. 
\end{proof}

Every full \cs-category $\Cf$ determines a subgroupoid, actually a total equivalence relation, in the (algebraic) Picard-Rieffel groupoid, with objects given by the diagonal \cs-algebras $\Cf_{AA}$, for all $A\in \Ob_\Cf$, and morphisms given by the equivalence classes, under isomorphism of bimodules, of ${}_\As\Cf_{\Bs}$. 
Such an association is functorial as specified by the following result, whose proof is now elementary. 
\begin{theorem}\label{th: pic}
Let $\Cf$ be a full \cs-category. Denote by $[\Cf_{AB}]$ the equivalence class of Hilbert \cs-bimodules that are isomorphic to the imprimitivity bimodule ${}_\As\Cf_\Bs$. 
Consider $[\Cf_{AB}]$, for all $A,B\in\Ob_\Cf$, as arrows in the (algebraic) Picard-Rieffel groupoid.  
The family 
\begin{equation*}
\Pic_\Cf:=\{[\Cf_{AB}] \ | \ A,B\in \Ob_\Cf\}, 
\end{equation*}
is a total equivalence relation (i.e.~a subgroupoid with one and only one arrow for every pair of objects) contained in the algebraic Picard-Rieffel groupoid. 

A $*$-functor\footnote{A $*$-functor $\Phi:\Cf\to\Df$ between C*-categories is just a functor (linear on each block $\Cf_{AB}$, $A,B\in \Ob_\Cf$) such that $\Phi(x^*)=\Phi(x)^*$ for all $x\in \Hom_\Cf$. } 
$\Phi:\Cf\to\Df$ between full \cs-categories that is bijective on objects uniquely determines an isomorphism $\Pic(\Phi):\Pic_\Cf\to\Pic_\Df$ of equivalence relations given by:
\begin{equation*}
\Pic(\Phi): [\Cf_{AB}]\mapsto [\Df_{\Phi_A\Phi_B}], \quad \forall A,B\in \Ob_\Cf,
\end{equation*}
where $\Phi:A\mapsto\Phi_A\in\Ob_\Df$ denotes the bijective action of the functor on the objects of $\Cf$. 
The map $\Pic$ is a functor from the category of object-bijective $*$-functors between small full \cs-categories into the category of (object bijective) groupoid homomorphisms between total equivalence relations contained in the algebraic Picard-Rieffel groupoid. 
\end{theorem}
An important tool related to these considerations is the ``linking algebra'' 
$\left[
\begin{smallmatrix}
\As & {}_\As M_\Bs \\
{}_\Bs M^*_\As& \Bs
\end{smallmatrix}
\right]$
of an imprimitivity bimodule ${}_\As M_\Bs$ 
as defined in L.~Brown-P.~Green-M.~Rieffel~\cite{BGR}, that could be seen as the enveloping \cs-algebra (see~\cite{GLR}) of a \cs-category with two objects. 

Since by~\cite[Theorem~1.1]{BGR} two unital \cs-algebras $\As,\Bs$ are Morita equivalent if and only if there exists another unital \cs-algebra $\Cs$ and two projections $p,q\in \Cs$ such that:
\begin{gather*}
p+q=1, \quad 
p\Cs p\simeq \As, \quad q\Cs q \simeq \Bs, \quad 
\cj{\Cs p \Cs}=\Cs, \quad \cj{\Cs q \Cs}=\Cs, 
\end{gather*}
and in this case there is a natural \cs-category with two objects with linking algebra 
$\left[
\begin{smallmatrix}
p\Cs p & q \Cs p \\
p\Cs q & q \Cs q
\end{smallmatrix}
\right]$, 
it is likely that every full C*-category can be seen as a ``strictification'' of a total equivalence relation in the ``weak'' Picard-Rieffel groupoid and hence that the functor $\Pic$ in theorem~\ref{th: pic} is surjective on objects. We will return to these considerations elsewhere.  

\bigskip

Following now~\cite{BCL,BCL2}, we say that a C*-category $\Cf$ is \emph{commutative} if all its diagonal blocks $\Cf_{AA}$ are commutative 
C*-algebras.

\medskip

When an imprimitivity bimodule is actually the bimodule ${}_\As\Cf_\Bs$ of morphisms $\Hom_\Cf(B,A)$ in a full commutative C*-category $\Cf$, much more can be said about the properties of the canonical isomorphisms of theorem~\ref{lemma: phiM} 
\begin{equation}\label{eq: phiba}
\phi_{BA}:=\phi_{{}_\As\Cf_\Bs}: \As\to \Bs.
\end{equation}
\begin{proposition}\label{prop: 3prop}
Let $\Cf$ be a full commutative C*-category, the family of canonical isomorphisms $(A,B)\mapsto \phi_{BA}$ associated to the imprimitivity bimodules ${}_\As\Cf_\Bs$ satisfies the following compatibility conditions:
\begin{gather}
\phi_{AA}=\iota_\As, \quad \forall A\in \Ob_\Cf, \label{eq: cm1} \\
\phi_{BA}=\phi_{AB}^{-1}, \quad \forall A,B\in \Ob_\Cf, \label{eq: cm2} \\
\phi_{CB}\circ\phi_{BA}=\phi_{CA}, \quad \forall A,B,C\in \Ob_\Cf. \label{eq: cm3}
\end{gather}
\end{proposition}
\begin{proof}
First of all, we note again that, for imprimitivity bimodules ${}_\As\Cf_\Bs$ of morphisms in a commutative full \cs-category, there is an explicit description of the inner products:
\begin{equation*}
\ip{x}{y}_\Bs:=x^*y, \quad {}_\As\ip{x}{y}:=yx^* \quad \forall x,y\in {}_\As\Cf_\Bs. 
\end{equation*}

Property~\eqref{eq: cm1} follows immediately from
\begin{equation*}
\phi_{AA}(a)=\sum_j\ip{w_j}{az_j}_\As=\sum_jw_j^*az_j=a\sum_j\ip{w_j}{z_j}_\As=a \quad \forall a\in {}_\As\Cf_\As.
\end{equation*}

To prove property~\eqref{eq: cm3}, let $w_j,z_j$ be finite families of elements in ${}_\As\Cf_\Bs$ and $x_k,y_k$  finite families of elements in ${}_\Bs\Cf_\Cs$ such that $\sum_j\ip{w_j}{z_j}_\Bs=1_\Bs$ and $\sum_k\ip{x_k}{y_k}_\Cs=1_\Cs$. By the definition of the canonical isomorphism \eqref{eq: phi-def}, we have:
\begin{gather*}
\phi_{BA}(a):=\sum_j\ip{w_j}{az_j}_\Bs \quad \forall a\in \As, \\
\phi_{CB}(b):=\sum_k\ip{x_k}{by_k}_\Cs \quad \forall b\in \Bs.
\end{gather*}
By direct calculation we see that the composition is given by:
\begin{align*}
\phi_{CB}\circ\phi_{BA}(a) & = \sum_k \ip{x_k}{\sum_j\ip{w_j}{az_j}_\Bs \,y_k}_\Cs \\
& = \sum_k\sum_j x_k^*w_j^*az_jy_k = \sum_k\sum_j (w_jx_k)^*a(z_jy_k).
\end{align*}
We only need to prove that the expression above is of the form $\sum_h \ip{u_h}{av_h}_\Cs$ for finite families of elements $u_h,v_h\in {}_\As\Cf_\Cs$, indexed by $h$, such that $\sum_h\ip{u_h}{v_h}_\Cs=1_\Cs$.

Now, the families of elements $w_jx_k$ and $z_jy_k$ satisfy exactly this property 
\begin{align*}
\sum_k\sum_j\ip{w_jx_k}{z_jy_k}_\Cs & =\sum_k\sum_j x_k^*w_j^*z_jy_k =  \sum_k\ip{x_k}{\sum_j\ip{w_j}{z_j}_\Bs y_k}_\Cs \\ 
& = \sum_k\ip{x_k}{1_\Bs y_k}_\Cs = 1_\Cs 
\end{align*}
and so we can define $u_{j,k}:=w_jx_k\in {}_\As\Cf_\Cs$ and $v_{j,k}:=z_jy_k\in {}_\As\Cf_\Cs$. 

Property~\eqref{eq: cm2} follows by direct application of equations~\eqref{eq: cm1} and~\eqref{eq: cm3}.
\end{proof}

\begin{proposition} \label{prop: omega-phi}
Let $\omega:\Cf\to \CC$ be a $*$-functor (i.e.~a functor such that $\omega(x^*)=\cj{\omega(x)}$, for all $x\in \Cf$) defined on the full commutative \cs-category $\Cf$.  For every pair of objects $A,B\in \Ob_\Cf$, we have
\begin{equation*}
\omega(\phi_{BA}(a))=\omega(a), \quad \forall a\in \Cf_{AA}. 
\end{equation*}
\end{proposition}
\begin{proof}
Consider the imprimitivity bimodule ${}_\As\Cf_\Bs$ and the associated canonical isomorphism $\phi_{BA}:\Cf_{AA}\to\Cf_{BB}$. For every $a\in \Cf_{AA}$, for any given finite families $w_j,z_j\in \Cf_{AB}$ such that $\sum_j\ip{w_j}{z_j}_\Bs=1_\Bs$, we know that $\phi_{BA}(a)=\sum_j \ip{w_j}{az_j}_\Bs$.  
Since $\omega:\Cf\to\CC$ is a $*$-functor, for all $a\in \Cf_{AA}$, we have:
\begin{align*}
\omega(\phi_{BA}(a)):&=\omega(\sum_j\ip{w_j}{az_j}_\Bs) = \sum_j\omega(\ip{w_j}{az_j}_\Bs) \\
&= \sum_j\omega(w_j^*az_j)=\sum_j\omega(w_j^*)\omega(a)\omega(z_j) \\ 
&= \omega(a) \sum_j \omega(w_j^*)\omega(z_j)=\omega(a)\sum_j\omega(w_j^*z_j) \\
&= \omega(a)\omega (\sum_j \ip{w_j}{z_j}_\Bs) =\omega(a)\omega(1_\Bs)=\omega(a). 
\end{align*}
\end{proof}

\section{Spectral Theorem for Imprimitivity Bimodules} \label{sec: bimodule}

Let $X_A$ and $X_B$ be two compact Hausdorff spaces and let $R_{BA}: X_A\to X_B$ be a homeomorphism between them. To every complex 
bundle $(E,\pi,R_{BA})$, over the graph of the homeomorphism $R_{BA}\subset X_A\times X_B$, we can naturally associate the set 
$\Gamma(R_{BA};E)$ of continuous sections of the bundle $E$, that turns out to be a symmetric bimodule over the commutative \cs-algebra 
$C(R_{BA};\CC)$ of continuous functions on the compact Hausdorff space $R_{BA}$. 

Considering now the pair of homeomorphisms
\begin{gather*}
\pi_A: R_{BA}\to X_A, \quad \pi_A: (x,y)\mapsto x, \\
\pi_B: R_{BA}\to X_B,  \quad \pi_B: (x,y)\mapsto y, 
\end{gather*}
we see that the set $\Gamma(R_{BA};E)$ becomes naturally a left module over $C(X_A;\CC)$ and a right module over $C(X_B;\CC)$ with the following left and right actions $f\cdot\sigma:=(f\circ\pi_A)\cdot \sigma$ and 
$\sigma\cdot g:=\sigma\cdot (g\circ\pi_B)$ or, in a more explicit form, for all $(x,y)\in R_{BA}$, $f\in C(X_A)$, $g\in C(X_B)$ and $\sigma\in \Gamma(R_{BA};E)$:
\begin{gather*}
f\cdot \sigma (x,y) := f(x)\sigma(x,y)=(f\circ\pi_A)(x,y) \cdot \sigma(x,y), \\
\sigma\cdot g(x,y) := \sigma(x,y)g(y)=\sigma(x,y) \cdot (g\circ\pi_B)(x,y). 
\end{gather*}
In the terminology of definition~\ref{def: twist}, this is the bimodule ${}_{\pi_A^\bullet}\Gamma(R_{BA},E)_{\pi_B^\bullet}$ obtained by twisting the symmetric $C(R_{BA})$-bimodule $\Gamma(R_{BA},E)$ by the isomorphism 
\hbox{$\pi_A^\bullet:C(X_A)\to C(R_{BA})$} on the left and by the isomorphism $\pi_B^\bullet:C(X_B)\to C(R_{BA})$ on the right. 

We say that ${}_{\pi_A^\bullet}\Gamma(R_{BA};E)_{\pi_B^\bullet}$ is the \emph{$C(X_A)$-$C(X_B)$-bimodule associated to the bundle $(E,\pi,R_{BA})$ over the homeomorphism $R_{BA}:X_A\to X_B$}.
Note that if $(E,\pi,R_{BA})$ is a Hermitian bundle over the homeomorphism $R_{BA}:X_A\to X_B$, then the bimodule ${}_{C(R_{BA})}\Gamma(R_{BA};E)_{C(R_{BA})}$ is a full symmetric Hilbert \cs-bimodule over $C(R_{BA})$ and, as in remark~\ref{rem: twist}, 
the associated bimodule ${}_{\pi_A^\bullet}\Gamma(R_{BA};E)_{\pi_B^\bullet}$ has a natural structure as a full Hilbert \cs-bimodule with inner products given by: 
\begin{gather*} 
{}_{C(X_A)}\ip{\sigma}{\rho}:=(\pi_A^\bullet)^{-1}(\ip{\sigma}{\rho}_{C(R_{BA})}), 
\quad \forall \sigma,\rho\in \Gamma(R_{BA};E), \\ 
\ip{\sigma}{\rho}_{C(X_B)}:=(\pi_B^\bullet)^{-1}(\ip{\sigma}{\rho}_{C(R_{BA})}), 
\quad \forall \sigma,\rho\in \Gamma(R_{BA};E).  
\end{gather*}
Furthermore the associated bimodule ${}_{\pi_A^\bullet}\Gamma(R_{BA};E)_{\pi_B^\bullet}$ is an imprimitivity bimodule if and only if ${}_{C(R_{BA})}\Gamma(R_{BA};E)_{C(R_{BA})}$ is an imprimitivity bimodule and this, by Serre-Swan theorem (see e.g.~\cite[Section~2.1.2]{BCL} and references therein), happens if and only if $(E,\pi,R_{BA})$ is a Hermitian line bundle. 

In this section, making use of the results in section~\ref{sec: imp-bim}, we prove, in the case of imprimitivity bimodules, a converse to the previous construction i.e.~that (up to isomorphism of bimodules) every imprimitivity Hilbert \cs-bimodule ${}_\As \Ms_\Bs$ over unital commutative 
\cs-algebras $\As$ and $\Bs$ actually arises as the bimodule associated to a Hermitian line bundle over a homeomorphism between the compact Hausdorff spaces $\Sp(\As)$ and $\Sp(\Bs)$. 

\begin{theorem}\label{th: impri}
Given an imprimitivity \cs-bimodule ${}_\As \Ms_\Bs$ over two commutative unital \cs-algebras $\As,\Bs$, there exists a Hermitian line bundle $(E,\pi,R_{BA})$, over the graph of a homeomorphism  $R_{BA}: X_A\to X_B$ between the two compact Hausdorff spaces $X_A:=\Sp(\As)$, $X_B:=\Sp(\Bs)$, whose associated $C(X_A)$-$C(X_B)$-bimodule ${}_{\pi_A^\bullet}\Gamma(R_{BA};E)_{\pi_B^\bullet}$, when twisted on the left by the Gel'fand transform isomorphism $\Gg_\As:\As\to C(\Sp(\As))$ and on the right by the Gel'fand isomorphism $\Gg_\Bs:\Bs\to C(\Sp(\Bs))$, becomes a bimodule ${}_{\pi_A^\bullet\circ\Gg_\As}\Gamma(R_{BA};E)_{\pi_B^\bullet\circ\Gg_\Bs}$ that is isomorphic, as an $\As$-$\Bs$-bimodule, to the initial Hilbert  \cs-bimodule ${}_\As\Ms_\Bs$.
\end{theorem}
\begin{proof}
By theorem~\ref{lemma: phiM}, we have a canonical isomorphism $\phi_\Ms:\As\to\Bs$. 
Using Gel'fand theorem, applied to the isomorphism $\phi_\Ms^{-1}:\Bs\to\As$ of unital \cs-algebras, we recover a homeomorphism $R_{BA}:=(\phi_M^{-1})^\bullet:X_A\to X_B$ between the two compact Hausdorff spaces $X_A:=\Sp(\As)$ and $X_B:=\Sp(\Bs)$. Furthermore we know that the Gel'fand transforms $\Gg_\As:\As\to C(X_A;\CC)$, $\Gg_\Bs:\Bs\to C(X_B;\CC)$ provide two isomorphisms of \cs-algebras. 

Consider now the set $\Rs\subset\As\times\Bs$ defined by
$\Rs:=\{(a,b)\in\As\times\Bs\ | \ b=\phi_\Ms(a) \}$  
and note that $\Rs$ has a natural structure of unital \cs-algebra with componentwise multiplication and norm defined by 
$\|(a,b)\|_\Rs:=\max\{\|a\|,\|b\|\}=\|a\|=\|b\|$. 
There are natural isomorphisms $\alpha:\Rs\to\As$ and $\beta:\Rs\to\Bs$ given by
\begin{equation*}
\alpha: (a,b)\mapsto a, \quad \beta:(a,b)\mapsto b, \quad \forall (a,b)\in \Rs,  
\end{equation*}
and they satisfy $\phi_\Ms=\beta\circ\alpha^{-1}$. 

Note also that the topological space $\Sp(\Rs)$ is canonically homeomorphic to $R_{BA}$. 
In fact, since $R_{BA}\circ(\alpha^{-1})^\bullet=(\phi^{-1}_\Ms)^\bullet\circ(\alpha^{-1})^\bullet =(\alpha\circ \beta^{-1})^\bullet\circ(\alpha^{-1})^\bullet=(\beta^{-1})^\bullet$, the function \hbox{$S:\omega\mapsto((\alpha^{-1})^\bullet(\omega),(\beta^{-1})^\bullet(\omega))$}, for $\omega\in \Sp(\Rs)$, takes values in $R_{BA}$ and being bijective continuous between compact Hausdorff spaces it is a homeomorphism.

We summarize the situation with the following commutative diagrams that might come helpful to visualize the several isomorphisms and homeomorphisms involved:
\begin{equation*}
\xymatrix{
\As \ar[d]_{\Gg_\As} & \Rs\ar[l]_\alpha \ar[r]^\beta \ar[d]_{\Gg_\Rs} & \Bs \ar[d]_{\Gg_\Bs} \\
C(X_A)  \ar[dr]_{\pi_A^\bullet} & \ar[l]_{\alpha^{\bullet\bullet}}C(\Sp(\Rs)) \ar[r]^{\beta^{\bullet\bullet}}& C(X_B) \ar[dl]^{\pi_B^\bullet} \\
 & C(R_{BA}) \ar[u]_{S^\bullet} & 
}\quad 
\xymatrix{
X_A \ar[dr]^{\alpha^\bullet}\ar[rr]^{R_{BA}}& & X_B \ar[dl]_{\beta^\bullet} \\
 & \Sp(\Rs)\ar[d]_S & \\
 & R_{BA} \ar[uul]^{\pi_A} \ar[uur]_{\pi_B} &  
}
\end{equation*}

Twisting (see definition~\ref{def: twist}) the bimodule ${}_\As\Ms_\Bs$ by $\alpha$ on the left and $\beta$ on the right, we obtain a Hilbert \cs-bimodule ${}_\alpha\Ms_\beta$ over $\Rs$ that is symmetric because
\begin{equation*}
(a,b)\cdot x = \alpha(a,b)x=ax=x\phi_\Ms(a)=x\beta(a,b)=x\cdot(a,b), \forall (a,b)\in \Rs. 
\end{equation*}
Twisting one more time ${}_\alpha\Ms_\beta$ with the isomorphism 
\begin{equation*}
\gamma:=\Gg_\Rs^{-1}\circ S^\bullet:C(R_{BA})\to \Rs, 
\end{equation*}
we get a symmetric Hilbert \cs-bimodule ${}_{\alpha\circ\gamma}\Ms_{\beta\circ\gamma}$ over the \cs-algebra $C(R_{BA})$. 
By a direct application of Serre-Swan theorem (see e.g.~\cite[Theorem~2.2]{BCL}), we see that there exists a Hermitian bundle $(E,\pi, R_{BA})$ over the compact Hausdorff space $R_{BA}$ such that there exists an isomorphism of $C(R_{BA})$-bimodules  $\Phi:{}_{\alpha\circ\gamma}\Ms_{\beta\circ\gamma}\to\Gamma(R_{BA};E)$. 
Since ${}_\As M_\Bs$ is an imprimitivity bimodule, so is  
${}_{\alpha\circ\gamma}\Ms_{\beta\circ\gamma}$ and hence $(E,\pi,R_{BA})$ is a Hermitian line bundle. 
Making use of proposition~\ref{prop: twist}, we have that the map $\Phi$ also becomes an isomorphism $\Phi:{}_\As M_\Bs \to {}_{(\alpha\circ\gamma)^{-1}}\Gamma(R_{BA};E)_{(\beta\circ\gamma)^{-1}}$ of Hilbert \cs-bimodules over $\As$ and $\Bs$. 
Since, by the diagram above, we have $(\alpha\circ\gamma)^{-1}=\pi_A^\bullet\circ \Gg_A$ and $(\beta\circ\gamma)^{-1}=\pi_B^\bullet\circ\Gg_B$, we finally obtain an isomorphism of left $\As$, right $\Bs$ Hilbert \cs-bimodules
\begin{equation*}
\Phi:{}_\As M_\Bs \to 
{}_{\pi_A^\bullet\circ \Gg_A}\Gamma(R_{BA};E)_{\pi_B^\bullet\circ\Gg_B}.
\end{equation*} 
\end{proof}

Note that the theorem says that for an imprimitivity bimodule ${}_\As M_\Bs$ over commutative unital \cs-algebras, the triple $(\Gg_\As,\Phi,\Gg_\Bs)$ provides an isomorphism, in the category of Hilbert \cs-bimodules, from the bimodule ${}_\As M_\Bs$ to the $C(X_A)$-$C(X_B)$-bimodule ${}_{\pi_A^\bullet}\Gamma(R_{BA};E)_{\pi_B^\bullet}$ associated to the Hermitian line bundle $(E,\pi,R_{BA})$ over the homeomorphism 
\hbox{$R_{BA}:X_A\to X_B$}. 
This means that $\Phi(axb)=\Gg_\As(a)\Phi(x)\Gg_\Bs(b)$, for all $x\in \Ms$, $a\in \As$ and $b\in \Bs$. 
The map $\Phi$ is essentially a ``canonical extension'' of the Gel'fand transform of the \cs-algebras $\As$ and $\Bs$ to the imprimitivity bimodule ${}_\As M_\Bs$  over them. 

\medskip

The above theorem is just the starting point for the development of a ``bivariant Serre-Swan equivalence'' and, more generally, a bivariant ``Takahashi duality'' (see e.g.~\cite[Section~2.1.2]{BCL} and references therein) for the category of  Hilbert \cs-bimodules over commutative \cs-algebras. This will be done elsewhere. 

Our spectral theorem, for imprimitivity bimodules over Abelian \cs-algebras, is dealing only with the representativity of a potential functor that, to every Hermitian line bundle $(E,\pi,R_{BA})$ over the graph of a homeomorphism $R_{BA}:X_A\to X_B$ between compact Hausdorff spaces, associates the imprimitivity bimodule ${}_{\pi_A^\bullet}\Gamma(R_{BA};E)_{\pi_B^\bullet}$ over the commutative \cs-algebras $C(X_A)$ and 
$C(X_B)$. To proceed further we have to provide a suitable notion of morphisms and define our functor on them. 

The above result is for now stated in the case of imprimitivity bimodules and hence it does not provide neither an answer to the problem of classifying, nor a geometric interpretation of general $C(X)$-$C(Y)$-bimodules for given compact Hausdorff spaces $X$ and $Y$. 
Warning the reader to take due care of some differences in notations and definitions, 
for some related results on the ``spectral theory'' of Hilbert \cs-bimodules,  
one may consult B.~Abadie-R.~Exel~\cite{AE}, H.~Bursztyn-S.~Waldmann~\cite{BW}, A.~Hopenwasser-J.~Peters-J.~Powers~\cite{HPP}, A.~Hopenwasser~\cite{H}, T.~Kajiwara-C.~Pinzari-Y.~Watatani~\cite{KPW}, P.~Muhly-B.Solel~\cite{MS}. 

In particular, B.~Abadie and R.~Exel~\cite[Proposition~1.9]{AE} proved that every imprimitivity \cs-bimodule over a commutative \cs-algebra $\As$ is always obtained from its right symmetrization by twisting on one side with a given automorphism $\theta$ and, in a more algebraic setting, a result of H.~Burzstyn-S.~Waldmann~\cite[Proposition~2.3]{BW} assures that if two imprimitivity bimodules ${}_\As M_\Bs$ and ${}_\As N_\Bs$ over the same commutative algebras are isomorphic as right modules, there is a unique isomorphism of the \cs-algebra $\Bs$  such that the bimodule $M$ is isomorphic to the twisting of $N$. 

Gathering together the above facts, in the special case of commutative full \cs-categories, we obtain the following result.
\begin{theorem}
Let $\Cf$ be a full commutative \cs-category.
Then for every pair of objects $A$ and $B$, one has:
\begin{itemize}
\item [-] ${}_\As\Cf_\Bs$ is an imprimitivity ${}_\As\Cf_\As$-${}_\Bs\Cf_\Bs$ bimodule.
That is, ${}_\As\Cf_\As$ and ${}_\Bs\Cf_\Bs$ are Morita equivalent and thus there is a canonical $*$-isomorphism implemented by $x^* y \mapsto y x^*$, $x,y \in {}_\As\Cf_\Bs$.
\item [-] ${}_\As\Cf_\Bs$ is the (non-symmetric) ${}_\As\Cf_\As$-${}_\Bs\Cf_\Bs$-bimodule of continuous sections of a Hermitian line bundle over the graph of the corresponding homeomorphism between the Gel'fand spectra of ${}_\As\Cf_\As$ and ${}_\Bs\Cf_\Bs$.
\end{itemize}
\end{theorem}

\bigskip

\noindent
\emph{Aknowledgements.}

We acknowledge the support provided by the Thai Research Fund, grant n.~RSA4780022.  
The main part of this work has been done in the two-year visiting period of R.~Conti to the Department of Mathematics of Chulalongkorn University.

P.~Bertozzini thanks the Department of Mathematics in Chulalongkorn University for the kind weekly hospitality during the period of preparation of this paper.

\end{document}